\documentclass{amsart}
\usepackage{hyperref}

\newcommand{\norm}[2][]{\|{#2}\|_{{#1}}}

\newcommand{\dd}{ { \mathrm{d}} }

\DeclareMathOperator*{\Tr}{\mathrm{Tr}}
\DeclareMathOperator*{\essup}{ess\,sup}

\newcommand{\cF}{{ \mathcal F}}
\newcommand{\cD}{{ \mathcal D}}

\newcommand{\HS}{\mathrm{HS}}

\newtheorem{thm}{Theorem} [section]

\newtheorem{lem} [thm] {Lemma}
\newtheorem{prop} [thm] {Proposition}
\theoremstyle {definition}

\numberwithin{equation}{section}

\begin{document}

\title[Euler approximation of the stochastic Allen-Cahn equation]{On
  the backward Euler approximation of the stochastic Allen-Cahn
  equation}

\author[M.~Kov\'acs]{Mih\'aly Kov\'acs}

\address{Department of Mathematics and Statistics,
  University of Otago, P.O.~Box 56, Dune\-din, New Zealand}

\email{mkovacs@maths.otago.ac.nz}

\author[S.~Larsson]{Stig Larsson}

\address{
  Department of Mathematical Sciences,
  Chalmers University of Technology and University of Gothenburg,
  SE--412 96 Gothenburg,
  Sweden}

\email{stig@chalmers.se}

\author[F.~Lindgren]{Fredrik Lindgren}

\address{
  Department of Mathematical Sciences,
  Chalmers University of Technology and University of Gothenburg,
  SE--412 96 Gothenburg,
  Sweden}

\email{fredrik.lindgren@chalmers.se}


\begin{abstract}
  We consider the stochastic Allen-Cahn equation perturbed by smooth
  additive Gaussian noise in a spatial domain with smooth boundary in
  dimension $d\le 3$, and study the semidiscretization in time of the
  equation by an implicit Euler method. We show that the method
  converges pathwise with a rate $O(\Delta t^{\gamma}) $ for any
  $\gamma<\frac12$. We also prove that the scheme converges uniformly
  in the strong $L^p$-sense but with no rate given.
\end{abstract}
\maketitle

\section{Introduction}
Let $\mathcal{D}\subset \mathbb{R}^d$, $d\le 3$, be a spatial domain with smooth boundary $\partial\cD$
and consider the stochastic partial differential equation written in the abstract It\^o form
\begin{equation}\label{eq:sac}
\dd u+Au\,\dd t+f(u)\,\dd t=\dd W,~t\in(0,T];\quad u(0)=u_0,
\end{equation}
where $\{W(t)\}_{t\ge 0}$ is an $L^2(\cD)$-valued $Q$-Wiener process
on a filtered probability space $(\Omega, \cF, \mathbb{P},
\{\cF_t\}_{t\ge 0})$ with respect to the normal filtration $
\{\cF_t\}_{t\ge 0}$.  We use the notation $H=L^2(\cD)$ with inner
product $\langle\cdot\,,\cdot\rangle$ and induced norm $\|\cdot\|$ and
$V=H^1_0(\cD)$. Moreover, $A\colon V\to V'$ denotes the linear
elliptic operator $Au=-\nabla\cdot(\kappa\nabla u)$ for $u\in V$,
where $\kappa(x)>\kappa_0>0$ is smooth. As usual we consider the
bilinear form $a\colon V\times V\to\mathbb{R}$ defined by $a(u,v) = (
Au,v)$ for $u,v\in V$, and $(\cdot\, ,\cdot)$ denotes the duality pairing
of $V'$ and $V$.  We denote by $\{E(t)\}_{t\ge 0}$ the analytic
semigroup in $H$ generated by the realization of $-A$ in $H$ with
$D(A)=H^2(\mathcal D)\cap H^1_0(\mathcal{D})$.  Finally, $f\colon
\cD_f\subset H\to H$ is given by $(f(u))(x)=F'(u(x))$, where
$F(s)=c(s^2-\beta^2)^2$ ($c>0$) is a double well potential. Note that
$f$ is only locally Lipschitz and does not satisfy a linear growth
condition. It does, however, satisfy a global one-sided Lipschitz
condition, which is a key property for proving uniform moment bounds.

We consider a fully implicit Backward Euler discretization of \eqref{eq:sac}
via the iteration
\begin{equation}\label{eq:be}
u^j-u^{j-1}+\Delta t\, A u^j+\Delta t\, f(u^j)=\Delta W^j,~j=1,2,\dots, N;
\quad u^0=u_0,
\end{equation}
where $\Delta t>0$. Note that this scheme is implicit
also in the drift term $f$. In return, the scheme preserves key
qualitative aspects of the solution of \eqref{eq:sac} such as moment
bounds.

The following two results constitute the main results of the
paper. For notation we refer to Section \ref{sec:prelim}. Let $N\in \mathbb{N}$, $T=N\Delta t$ and
$t_n=n\Delta t$, $n=1,2,\dots, N$. In
Theorem~\ref{thm:pwc} (pathwise convergence) we show that if
$\|A^{\frac12+\varepsilon} Q^{\frac12}\|_{\HS}<\infty$ for some small
$\varepsilon>0$, $\mathbb{E}\|u_0\|_1^2<\infty$, and $0\le \gamma
<\frac12$, then there are finite random variables $K\ge 0$ and $\Delta t_0>0$ such
that, almost surely,
$$
\sup_{t_n\in [0,T]}\|u(t_n)-u^n\|\le K \Delta t^\gamma,\quad \Delta t\le \Delta t_0.
$$
In Theorem~\ref{thm:strc} (strong convergence) we prove
that if $p\ge 1$ and $\mathbb{E}\|u_0\|_1^{2p}<\infty$, then
$$
\lim_{\Delta t\to 0}\mathbb{E}\sup_{t_{n}\in [0,T]}\|u(t_n)-u^n\|^p= 0.
$$

Since the method of proof uses a priori bounds obtained via energy
arguments together with a pathwise error analysis based on the mild
formulation of the equation, a strong rate cannot be obtained via this
line of argument. We would like to point out that the strong
convergence of the Backward Euler scheme is somewhat surprising given
the superlinearly growing character of $f$, see also the discussion in
\cite{Jentzen1}. We do not know of any results where strong
convergence results, with or without rates, are obtained for a
time-discretization scheme for an SPDE with non-global Lipschitz
nonlinearity without linear growth (for SODEs, we refer to \cite{HMS}).  There are many
results on pathwise and strong convergence of the
Backward Euler scheme (usually explicit in the drift term $f$) under
global Lipschitz conditions (or local Lipschitz with linear growth conditions), see, for example,
\cite{CoxVN,GyM,Haus1,Haus2} and the references therein. For
non-global Lipschitz nonlinearities the relatively recent method
developed in \cite{Jentzen1} uses a scheme which is based on the mild
formulation of the SPDE. This is also employed, for example, in
\cite{Blomker_et_al}. In that setting pathwise error estimates are
derived but strong convergence results would be rather difficult to
obtain as the method loses the information about the one-sided
Lipschitz condition on $f$, which can only be exploited in a
variational or weak solution approach. We also mention \cite{P2001}
where convergence in probability is obtained without global Lipschitz
conditions for the Backward Euler scheme.

Spatial pathwise convergence results for certain semilinear SPDEs with non-global Lipschitz $f$ without linear growth are obtained in
\cite{Blomker_Jentzen,Blomker_et_al}, both using spectral Galerkin
approximation. Concerning spatial strong convergence we only know of
\cite{Liu-1} and \cite{SS}, both with rates, based on a spectral Galerkin method and a
finite difference method, respectively. In the latter two papers the authors use
energy type arguments, and hence they can fully exploit the one-sided
Lipschitz character of $f$.

Finally, we would like to note that \eqref{eq:be} is also referred to
as Rothe's method. Since we can prove both pathwise and strong
convergence, one can set up a nonlinear wavelet-based adaptive
algorithm to solve the elliptic equation in each time-step and obtain
a implementable scheme, which converges both path-wise and strongly
in a similar way as in \cite{Dahlkeetal2} and \cite{KLU}.

The paper is organized as follows. In Section \ref{sec:prelim} we
collect frequently used results from infinite dimensional analysis and
introduce some notation. In Section \ref{sec:reg} we discuss the
spatial Sobolev regularity of the solution and the H\"older regularity
in time. In Section \ref{sec:apriori} we prove maximal type $p$-th
moment bounds on $u^n$ (Propositions \ref{prop:ul} and
\ref{prop:ulp}), which are in fact the exact analogues of the ones on
$u(t)$ (Proposition \ref{prop:eb}). Here we highlight that for $p=2$
the bounds only grow linearly in $T$, while for $p>2$ exponentially
because of a Gronwall argument. In Section \ref{sec:mr} we state and
prove the main results of the paper on the convergence of
\eqref{eq:be}. An important part of the proof is a maximal type error
estimate for the linear part (Proposition \ref{prop:wa}), where we
employ a discrete version of the celebrated factorization method.

\section{Preliminaries}\label{sec:prelim}
Throughout the paper we will use various norms for linear operators on
a Hilbert space. We denote by $\mathcal{L}(H)$, the space of bounded
linear operators on $H$ with the usual operator norm denoted by
$\|\cdot\|$. If for a positive semidefinite operator $T\colon H\to H$, the
sum
$$
\Tr T:=\sum_{k=1}^\infty\langle Te_k,e_k\rangle<\infty
$$
for an orthonormal basis (ONB) $\{e_k\}_{k\in \mathbb{N}}$ of $H$,
then we say that $T$ is trace-class. In this case $\Tr T$, the trace
of $T$, is independent of the choice of the ONB. If for an operator
$T\colon H\to H$, the sum
$$
\|T\|_{\HS}^2:=\sum_{k=1}^\infty\|Te_k\|^2<\infty
$$
for an ONB $\{e_k\}_{k\in \mathbb{N}}$ of $H$, then we say that $T$ is
Hilbert-Schmidt and call $\|T\|_{\HS}$ the Hilbert-Schmidt norm of
$T$.  The Hilbert-Schmidt norm of $T$ is independent of the choice of
the ONB. We have the following well-known properties of the trace and
Hilbert-Schmidt norms, see, for example, \cite[Appendix C]{DPZ},
\begin{align}
\label{eq:hs}
\|T\|&\le \|T\|_{\HS},\quad \|TS\|_{\HS}\le \|T\|_{\HS}\|S\|,
\quad\|ST\|_{\HS}\le \|S\|\,\|T\|_{\HS},
\\
\label{eq:tr}
\Tr Q&=\|Q^{\frac12}\|_{\HS}^2=\|T\|^2_{\HS}=\|T^*\|_{\HS}^2,
\quad\text{ if $Q=TT^*$.}
\end{align}

Next, we introduce fractional order spaces and norms.  It is well
known that our assumptions on $A$ and on the spatial domain
$\mathcal{D}$ imply the existence of a sequence of nondecreasing
positive real numbers $\{\lambda_k\}_{k\geq 1}$ and an orthonormal
basis $\{e_k\}_{k\geq 1}$ of $H$ such that
\begin{equation*}\label{eq:spectral}
Ae_k = \lambda_k e_k, \quad \lim_{k\rightarrow +\infty} \lambda_k = +\infty.
\end{equation*}
Using the spectral functional calculus for $A$ we introduce the fractional powers $A^s$, $s
\in \mathbb{R}$, of $A$ as
\begin{equation*}\label{eq:fp}
A^s v=\sum_{k=1}^{\infty}\lambda_k^s(v,e_k)e_k,\quad
D(A^s)=\Big\{v\in H:\|A^sv\|^2=\sum_{k=1}^{\infty}\lambda_k^{2s}(v,e_k)^2<\infty\Big\}
\end{equation*}
and spaces $\dot{H}^\beta=D(A^{\beta/2})$ with inner product $\langle
u , v\rangle_{\beta}=\langle A^{\frac{\beta}{2}}u ,
A^{\frac{\beta}{2}}v\rangle$ and induced norms
$\norm[\beta]{v}=\norm{A^{\beta/2} v}$. It is well-known that if $0\le
\beta < 1/2$, then $\dot{H}^\beta=H^\beta$ and if $1/2<\beta\le 2$,
then $\dot{H}^\beta=\{u\in H^\beta:u|_{\partial \mathcal{D}}=0\}$,
where $H^\beta$ denotes the standard Sobolev space of order $\beta$.

We recall the fact that the semigroup $\{E(t)\}_{t\ge 0}$ generated
by $-A$ is analytic and therefore it follows from \cite[Theorem
6.13]{Pazy} that for $t> s > 0$,
\begin{align}
&\|A^{\beta}E(t)v\|\le Ct^{-\beta}\|v\|,\quad\beta\ge 0,\label{eq:anal0}\\
\label{eq:anal}
&\|A^{\beta}(E(t)-E(s))v\|
\le Cs^{-(\beta+\gamma)}|t-s|^{\gamma+\rho}\|A^{\rho}v\|,
~\beta \ge 0,~0\le \gamma+\rho \le 1.
\end{align}

We will also use It\^o's Isometry and the Burkholder-Davies-Gundy
inequality for It\^o-integrals of the form $\int_0^t\langle
\eta(s),\dd \tilde{W}(s)\rangle$, where $\tilde{W}$ is a
$\tilde{Q}$-Wiener process. For this kind of integral, It\^o's
Isometry, \cite[Proposition 4.5]{DPZ} reads as
\begin{equation}\label{eq:rito}
\mathbb{E}\left|\int_0^t\langle \eta(s),\dd
  \tilde{W}(s)\rangle\right|^2
=\mathbb{E}\int_0^t\|\tilde{Q}^{\frac12}\eta(s)\|^2\,\dd s,
\end{equation}
and the Burkholder-Davies-Gundy inequality, \cite[Lemma 7.2]{DPZ}, takes the form
\begin{equation}\label{eq:rbdg}
\mathbb{E}\sup_{t\in [0,t_0]}
\left|\int_0^t\langle \eta(s),\dd \tilde{W}(s)\rangle\right|^p
\le C_p \mathbb{E}\left(\int_0^{t_0}\|\tilde{Q}^{\frac12}\eta(s)\|^2\,\dd s\right)^{\frac{p}{2}}, ~p\ge 2.
\end{equation}

Finally, if $Y$ is an $H$-valued Gaussian random variable with
covariance operator $\tilde{Q}$, then, by \cite[Corollary 2.17]{DPZ},
we can bound its $p$-th moments via its covariance operator as
\begin{equation}\label{eq:es1a}
\mathbb{E}\|Y\|^{2p}\le C_p (\mathbb{E}\|Y\|^2)^{p}
=C_p (\Tr \tilde{Q})^{p}=\|\tilde{Q}^{\frac12}\|_{\HS}^{2p}.
\end{equation}

\section{Regularity of the solution}\label{sec:reg}

The following existence, uniqueness, and regularity result can
essentially be found in \cite[Example 3.5]{Liu} for
$\mathcal{D}=[0,1]$, where it is stated with $\essup$ instead of
$\sup$ for the second term in \eqref{eq:esup}. It is remarked there,
\cite[Remark 3.4]{Liu}, that the result can be proved in higher
dimensions by using \cite[Example 3.2]{LR}, where domains with smooth
boundary are considered. Finally, by \cite[Theorem 1.1]{Liu0}, the
$\essup$ can be replaced by $\sup$ in the second term as stated below
in \eqref{eq:esup}. We also  note that for the equation considered in
this paper, this result can be obtained by using the deterministic
Ljapunov functional $J(u)=\|\nabla u\|_1^2+\int_{{\mathcal
    D}}F(u)\,\dd x$ and It\^o's formula in a way analogous to
\cite[Theorem 3.1 and Corollary 3.2]{KML}, see also \cite{Prato}.

For the definition of \emph{variational solution} we refer to
\cite[Definition 4.2.1]{PR}.

\begin{prop}\label{prop:eb}
If $\|A^{\frac12}Q^{\frac12}\|_{\HS}<\infty$ and
$\mathbb{E}\|u_0\|_1^p<\infty$ for some $p\ge 2$, then there is a
unique variational solution $u$ of \eqref{eq:sac}. Furthermore,
there is $C_T>0$ such that
\begin{equation}\label{eq:esup}
\mathbb{E}\sup_{t\in [0,T]}\|u(t)\|^p+\mathbb{E}\sup_{t\in [0,T]}\|u(t)\|_{1}^p\le C_T.
\end{equation}
\end{prop}

In this case, $u$ is also a mild solution, see
\cite[Proposition F.0.5 and Remark F.0.6]{PR}; that is, $u$ satisfies
the integral equation
\begin{equation}\label{eq:mild}
u(t)=E(t)u_0+\int_0^tE(t-s)f(u(s))\,\dd s+W_A(t),\quad t\in[0,T],
\end{equation}
almost surely, where the stochastic convolution $W_A$ is defined by
$$
W_A(t)=\int_0^tE(t-s)\,\dd W(s).
$$
This ultimately follows from the fact that the noise is additive trace class and that, by Sobolev's inequality,
\begin{equation}\label{eq:lg}
\|f(u(t))\|\le C (\|u(t)\|+\|u(t)\|_{L^6}^3)
\le C (\|u(t)\|+\|u(t)\|_{1}^3),
\end{equation}
which is bounded almost surely for $t\in [0,T]$ by Proposition \ref{prop:eb}. Note that here, in order to be able to use Sobolev's
inequality, it is crucial that $d\le 3$ and that the nonlinearity
$f$ is at most cubic.

Next we look at the pathwise H\"older regularity of $u$. First we
consider the stochastic convolution $W_A$.
\begin{lem}\label{lem:wlip}
Let $0<\beta\le 1$,
$\|A^{\frac{\beta-1}{2}}Q^{\frac12}\|_{\HS}<\infty$ and $p>
\frac{2}{\beta}$. Then, there is a nonnegative real random variable
$K$ with $\mathbb{E}K^p<\infty$ such that, almost surely,
$$
\sup_{t\neq s\in [0,T]}\frac{\|W_A(t)-W_A(s)\|}{|t-s|^{\gamma}}\le K
\quad\text{for $0\le\gamma<\frac{\frac{\beta p}{2}-1}{p}$.}
$$
\end{lem}
\begin{proof}
Let $t> s\ge 0$. Note that the stochastic integrals below are
Gaussian random variables and hence we can use \eqref{eq:es1a} to
bound their $p$-th moments. Therefore,
\begin{equation*}
\begin{split}
\mathbb{E}\|W_A(t)-W_A(s)\|^p
& \le C_p\mathbb{E}\left\|\int_s^t E(t-\sigma)\, \dd W(\sigma)\right\|^p
\\ &\quad +C_p\mathbb{E}\left\|\int_0^s E(t-\sigma)-E(s-\sigma)\, \dd W(\sigma)\right\|^p\\
& \le C_p\left(\int_s^t \|E(t-\sigma)Q^{\frac12}\|^2_{\HS}\, \dd \sigma\right)^{\frac{p}{2}}
\\ &\quad +C_p\left( \int_0^s \|E(t-\sigma)-E(s-\sigma)Q^{\frac12}\|_{\HS}^2\, \dd\sigma\right)^{\frac{p}{2}}\le C|t-s|^{\frac{\beta p}{2}},
\end{split}
\end{equation*}
where the last inequality is shown in the proof of \cite[Theorem
4.2]{KLU}. Then the statement follows from Kolmogorov's criterion,
see, for example, \cite[Theorem 1.4.1]{Kunita}.
\end{proof}

With the above preparations, we now prove the H\"older continuity of
$u$. Note that the result is suboptimal compared to the corresponding
result for $W_A$ in Lemma~\ref{lem:wlip}, which requires only
$\beta=1$ to get the same H\"older exponent, while here we assume
$\beta=2$.  This is a consequence of the fact that we use the mild
formulation here and hence cannot exploit the one-sided
Lipschitz condition on $f$ but only its cubic growth.

\begin{prop}\label{propo:ulip}
  Let $\|A^{\frac{1}{2}}Q^{\frac12}\|_{\HS}<\infty$,
  $\mathbb{E}\|u_0\|_1^2<\infty$ and $T>0$. Then, for all
  $\gamma\in[0,\frac12)$, there is a finite nonnegative random
  variable $K$ such that, almost surely,
$$
\sup_{t\neq s\in [0,T]}\frac{\|u(t)-u(s)\|}{|t-s|^{\gamma}}\le K.
$$
\end{prop}
\begin{proof}
Let $T>0$,  $0\le s< t\le T$, and $0\le \gamma<\frac{1}{2}$.
We use the mild formulation \eqref{eq:mild} to represent $u(t)-u(s)$ as follows:
\begin{align*}
u(t)-u(s)&=(E(t)-E(s))u_0
+\int_s^tE(t-r)f(u(r))\,\dd r\\
& \quad +\int_0^s(E(t-r)-E(s-r)f(u(r))\,\dd r
+W_A(t)-W_A(s).
\end{align*}
The estimate in \eqref{eq:anal}, with $\beta=\gamma=0$ and
$\rho=\frac12$, implies that $\|(E(t)-E(s))u_0\|\le
C|t-s|^{\frac12}\|u_0\|_1$. The second term can be bounded, using
 Proposition \ref{prop:eb} and \eqref{eq:lg},
\begin{align*}
\left\|\int_s^tE(t-r)f(u(r))\right\|&\le\int_s^t\|E(t-r)\|\,\|f(u(r))\|\,\dd r\\
&\le C |t-s| \sup_{r\in [0,T]}(\|u(r)\|+\|u(r)\|_1^3)\le K |t-s|.
\end{align*}

In a similar fashion, using this time \eqref{eq:anal} with
$\beta=\rho=0$ and $\frac12\le\gamma<1$,
\begin{align*}
&\left\|\int_0^s(E(t-r)-E(s-r))f(u(r))\,\dd r\right\|\\
&\qquad\le \int_0^s\|(E(t-r)-E(s-r))\|\,\dd r\sup_{r\in [0,T]}(\|u(r)\|+\|u(r)\|_1^3)\\
&\qquad\le K |t-s|^\gamma \int_0^sr^{-\gamma}\dd r\le KT^{1-\gamma}|t-s|^{\gamma}.
\end{align*}

Finally, we note that
$\|Q^{\frac12}\|_{\HS}\le C\|A^{\frac{1}{2}}Q^{\frac12}\|_{\HS}<\infty$ by \eqref{eq:hs} as $A^{-\frac12}\in
\mathcal{L}(H)$, so that we can use Lemma \ref{lem:wlip} with
$\beta=1$ to conclude the proof.
\end{proof}

\section{A priori moment bounds}\label{sec:apriori}

Our first result bounds the second moment of the Euler iterates in
\eqref{eq:be}. The proof uses a kind of bootstrapping argument and as a
result we avoid Gronwall's lemma. Therefore, we are able
to obtain bounds that only grow linearly with $T$ instead of
exponentially. Since these bounds will be used in the Gronwall step in
the pathwise convergence analysis, the constants appearing there will
grow exponentially with time instead of double-exponentially.  We have
to use test functions in the energy arguments below that are different
from the ones used in the deterministic setting, for example in
\cite{Stig}, because of the presence of a non-differentiable right
hand side. This ultimately forces the choice of a scheme implicit also
in the drift in order to be able to use the one-sided Lipschitz
property of $f$.

\begin{prop}\label{prop:ul}
Let $I_N=\{1,2,\dots, N\}$ and $T=N\Delta t$. If $\|A^{\frac12}Q^{\frac12}\|_{\HS}<\infty$ and $\mathbb{E}\|u_0\|_1^2<\infty$, then there is $C>0$ independent of $T$ such that
$$
\mathbb{E}\sup_{l\in I_N}\|u^l\|^2+\mathbb{E}\sup_{l\in I_N}\|u^l\|^2_1\le C(1+T).
$$
\end{prop}
\begin{proof}
  First note that it is enough to bound the second term on the left
  hand side since $\|\cdot\|\le C\|\cdot\|_1$.  Taking the inner
  product of \eqref{eq:be} with $u^j$, we get
$$
\langle u^j-u^{j-1}, u^j\rangle
+\Delta t\, \|u^j\|_{1}^2
+\Delta t\, \langle f(u^j),u^j\rangle
=\langle \Delta W^j, u^j\rangle.
$$
Using the identity $\langle
x-y,x\rangle=\frac{1}{2}(\|x\|^2-\|y\|^2)+\frac12\|x -y\|^2$ and the
fact that for some $C>0$ we have $sf(s)\ge -C$ for all $s\in
\mathbb{R}$ we get
\begin{align*}
&\frac{1}{2}(\|u^j\|^2-\|u^{j-1}\|^2)+\frac12\|u^j -u^{j-1}\|^2+\Delta t\, \|u^j\|_{1}^2 \\
&\qquad \le C\Delta t +\langle \Delta W^j, u^j-u^{j-1}\rangle+\langle \Delta W^j, u^{j-1}\rangle.
\end{align*}
Using a kick back with the second term on the right  and summing from $1$ to $n$ ($1\le n\le N$) gives
\begin{align*}
&\|u^n\|^2+\sum_{j=1}^n\|u^j -u^{j-1}\|^2+\Delta t\, \sum_{j=1}^n\|u^j\|_{1}^2 \\
&\qquad \le C\left(T+\|u_0\|^2+\sum_{j=1}^n\left(\|\Delta W^j\|^2+\langle \Delta W^j, u^{j-1}\rangle\right)\right).
\end{align*}
Taking expectation, using that $\Delta W^j$ is Gaussian with
covariance operator $\Delta t\, Q$ and hence $\mathbb{E}\|\Delta
W^j\|^2=\Delta t\, \Tr Q=\Delta t\, \|Q^{\frac12}\|^2_{\HS}$, and that
$\mathbb{E}\sum_{j=1}^n\langle \Delta W^j, u^{j-1}\rangle=0$, we
conclude
\begin{equation}\label{eq:iter1}
\begin{aligned}
&\mathbb{E}\left(\|u^n\|^2+ \sum_{j=1}^n\|u^j -u^{j-1}\|^2+\Delta t\, \sum_{j=1}^n\|u^j\|_{1}^2 \right)\\
&\qquad \le C\left(T+\mathbb{E}\|u_0\|^2+T\|Q^{\frac12}\|^2_{\HS}\right).
\end{aligned}
\end{equation}
Next, we take the inner product of \eqref{eq:be} with $Au^j$ and
obtain similarly as above
\begin{align*}
&\frac{1}{2}(\|u^j\|_{1}^2-\|u^{j-1}\|_{1}^2)+\frac{1}{2}\|u^j-u^{j-1}\|_{1}^2+\Delta t\, \|u^j\|_{2}^2+\Delta t\,\langle f(u^j),u^j\rangle_{1} \\
&\qquad = \langle \Delta W^j, u^j\rangle_{1}.
\end{align*}
Since $f'(s)\ge -C$, we have
\begin{equation*}
\langle f(u^j),u^j\rangle_{1}=\langle \nabla f(u^j),\nabla u^j\rangle
=\langle  f'(u^j)\nabla u^j,\nabla u^j\rangle
\ge -C\|u^j\|_1^2.
\end{equation*}
Hence,
\begin{align*}
&\frac{1}{2}(\|u^j\|_{1}^2-\|u^{j-1}\|_{1}^2)+\frac{1}{2}\|u^j-u^{j-1}\|_{1}^2+\Delta t\,\|u^j\|_{2}^2\\
&\qquad \le \Delta t\, C\|u^j\|_{1}^2+\langle\Delta W^j, u^j-u^{j-1}\rangle_{1}+\langle \Delta W^j, u^{j-1}\rangle_{1}.
\end{align*}
Thus, using a kick back with the second term, we obtain
\begin{align}\label{eq:1norm}
  \begin{split}
&\|u^l\|_{1}^2+\sum_{j=1}^l\|u^j-u^{j-1}\|_1^2+\Delta t\,\sum_{j=1}^l\|u^j\|^2_2\\
&\qquad\le C\left(\|u_0\|^2_1+\sum_{j=1}^l\left(\Delta t\,\|u^j\|^2_1
+\| \Delta W^j\|_1^2 +\langle \Delta W^j, u^{j-1}\rangle_{1}\right)\right).
  \end{split}
\end{align}
Therefore,
\begin{equation}\label{eq:iter2}
\begin{aligned}
&\mathbb{E}\sup_{l\in I_N}\left(\|u^l\|_{1}^2+\sum_{j=1}^l\|u^j-u^{j-1}\|_1^2+\Delta t\, \sum_{j=1}^l\|u^j\|^2_2\right)\\
&\quad\le C\mathbb{E}\|u_0\|^2_1+C\mathbb{E}\sup_{l\in I_N}\left(\sum_{j=1}^l\left(\Delta t\,\|u^j\|^2_1+\|\Delta W^j\|_1^2 +\langle\Delta W^j, u^{j-1}\rangle_{1}\right)\right)\\
&\quad\le C\mathbb{E}\|u_0\|^2_1+C\mathbb{E}\left(\sum_{j=1}^N\big(\Delta t\,\|u^j\|^2_1+\|\Delta W^j\|_1^2\big)\right)+C\mathbb{E}\sup_{l\in I_N}\sum_{j=1}^l\langle\Delta W^j, u^{j-1}\rangle_{1}.
\end{aligned}
\end{equation}
Since $A^{\frac12}\Delta W^j$ is a Gaussian random variable with
covariance operator $$\tilde{Q}:=\Delta t\,
A^{\frac12}Q^{\frac12}(A^{\frac12}Q^{\frac12})^*,$$ it follows, by
\eqref{eq:tr}, that
\begin{align*}
\mathbb{E}\|\Delta W^j\|_1^2=\Delta t\,\Tr\tilde Q=\Delta t\, \|A^{\frac12}Q^{\frac12}\|^2_{\HS}.
\end{align*}
Next note that $\sum_{j=1}^l\langle\Delta W^j, u^{j-1}\rangle_{1}$ is
an It\^o integral the form $\int_0^{t_l}\langle \eta(t),\dd
A^{\frac12}W(t)\rangle$, where $\eta$ is a piecewise continuous
process, and hence also a martingale.  Then, using H\"older's
inequality, the martingale inequality \cite[Theorem 3.8]{DPZ}, It\^o's
Isometry \eqref{eq:rito}, \eqref{eq:hs}, \eqref{eq:tr}, and
\eqref{eq:iter1},
\begin{align*}
&\left(\mathbb{E}\sup_{l\in I_N}\sum_{j=1}^l\langle \Delta W^j, u^{j-1}\rangle_{1}\right)^2\le \mathbb{E}\sup_{l\in I_N}\left(\sum_{j=1}^l\langle \Delta W^j, u^{j-1}\rangle_{1}\right)^2\\
&\qquad\le 4 \sup_{l\in I_N} \mathbb{E}\left(\sum_{j=1}^l\langle \Delta W^j, u^{j-1}\rangle_{1}\right)^2=4\mathbb{E}\Delta t\,\sum_{j=1}^N\|\tilde{Q}^{\frac12}A^{\frac12}u^{j-1}\|^2\\
&\qquad\le 4\|\tilde{Q}^{\frac12}\|^2\Delta t\, \sum_{j=1}^N\mathbb{E}\|u^{j-1}\|_1^2\le 4  \|\tilde{Q}^{\frac12}\|_{\HS}^2 \Delta t\,\sum_{j=1}^N\mathbb{E}\|u^{j-1}\|_1^2\\
&\qquad\le C\|A^{\frac12}Q^{\frac12}\|_{\HS}^2   (T+\mathbb{E}\|u_0\|^2+T\|Q^{\frac12}\|^2_{\HS}).
\end{align*}
Therefore, by \eqref{eq:iter2}, using also \eqref{eq:iter1}, we conclude that
\begin{align*}
\mathbb{E}\sup_{l\in I_N}\left(\|u^l\|_{1}^2+\sum_{j=1}^l\|u^j-u^{j-1}\|_1^2+\Delta t\,\sum_{j=1}^l\|u^j\|^2_2\right)
\le C(1+T)
\end{align*}
and the proof is complete.
\end{proof}
When proving strong convergence, even without rate, one needs bounds
on higher moments of the time discretization. This will be achieved
via a discrete Gronwall inequality, resulting in a bound that grows
exponentially with time. However, since our approach does not provide
rates for the strong error, this is not a major drawback. Note also,
that this result is the exact time-discrete analogue of the bounds on
the solution from Proposition \ref{prop:eb}.

\begin{prop}\label{prop:ulp}
  Let $p\ge 2$, $I_n=\{1,2,\dots, n\}$, $1\le n\le N$, and $ T=N\Delta t$. If
  $\|A^{\frac12}Q^{\frac12}\|_{\HS}<\infty$,
  $\mathbb{E}\|u_0\|_1^p<\infty$, and $T^{p-1}\Delta t\le \frac12$,
  then
$$
\mathbb{E}\sup_{l\in I_n}\|u^l\|^p+\mathbb{E}\sup_{l\in I_n}\|u^l\|^p_1\le C(T,p,u_0).
$$
\end{prop}
\begin{proof}
As noted in the proof of the previous proposition it is enough to bound the second term on the left hand side.
We start from \eqref{eq:1norm} and take the $p$th power of both sides for $p\ge 1$ to get
\begin{equation*}
\begin{aligned}
\|u^l\|_{1}^{2p}&\le C\left(\|u_0\|^{2p}_1+\Big(\sum_{j=1}^l\Delta t\, \|u^j\|^2_1\Big)^p\right.\\
&\quad+\left.\Big(\sum_{j=1}^l\| \Delta W^j\|_1^2\Big)^{p} +\Big(\sum_{j=1}^l\langle \Delta W^j, u^{j-1}\rangle_{1}\Big)^p\right)\\
& \le C\left(\|u_0\|^{2p}_1+\Delta t^{p-1}l^{p-1}\Delta t\, \sum_{j=1}^l\|u^j\|^{2p}_1\right.\\
&\quad \left. +l^{p-1}\sum_{j=1}^l\| \Delta W^j\|_1^{2p} +\Big(\sum_{j=1}^l\langle \Delta W^j, u^{j-1}\rangle_{1}\Big)^p\right).
\end{aligned}
\end{equation*}
Therefore,
\begin{equation}\label{eq:es1}
\begin{aligned}
\mathbb{E}\sup_{l\in I_n}\|u^l\|^{2p}_1&\le C \left(\mathbb{E}\|u_0\|^{2p}_1+T^{p-1}\Delta t\,\sum_{j=1}^n\mathbb{E}\sup_{l\in I_j}\|u^l\|^{2p}_1\right.\\
&\quad+ \left. n^{p-1}\sum_{j=1}^n\mathbb{E}\| \Delta W^j\|_1^{2p}+\mathbb{E}\sup_{l\in I_n}\Big(\sum_{j=1}^l\langle \Delta W^j, u^{j-1}\rangle_{1}\Big)^p\right).
\end{aligned}
\end{equation}
Next, we bound the last two terms in \eqref{eq:es1}. We already noted that $A^{\frac12}\Delta W^j$ is a Gaussian random variable with covariance operator $\tilde{Q}=\Delta t\, A^{\frac12}Q^{\frac12}(A^{\frac12}Q^{\frac12})^*$. Hence we use \eqref{eq:tr} and \eqref{eq:es1a} to bound its $2p$-th moment as
\begin{equation*}
\mathbb{E}\| \Delta W^j\|_1^{2p}\le C_p (\Tr\tilde{Q})^p=C_p\Delta t^p\|A^{\frac12}Q^{\frac12}\|^{2p}_{\HS}.
\end{equation*}
Therefore, it follows that
\begin{equation}\label{eq:es2}
n^{p-1}\sum_{j=1}^n\mathbb{E}\| \Delta W^j\|_1^{2p}\le C_pn^{p-1}\Delta t^p\sum_{j=1}^n\|A^{\frac12}Q^{\frac12}\|^{2p}_{\HS}\le C_pT^p\|A^{\frac12}Q^{\frac12}\|^{2p}_{\HS}.
\end{equation}
For the last term in \eqref{eq:es1} we use the Burkholder-Davies-Gundy inequality \eqref{eq:rbdg}, \eqref{eq:hs}, and \eqref{eq:tr} to conclude that
\begin{equation}\label{eq:es3}
\begin{aligned}
&\mathbb{E}\sup_{l\in I_n}\left(\sum_{j=1}^l\langle \Delta W^j, u^{j-1}\rangle_{1}\right)^p\le C_p\mathbb{E}\left( \Delta t\,\sum_{j=1}^n \|\tilde{Q}^{\frac12}A^{\frac12}u^{j-1}\|^2\right)^{p/2}\\
&\qquad\le C \|\tilde{Q}^{\frac12}\|^p \Delta t^{p/2} n^{p/2-1} \sum_{j=1}^n \mathbb{E}\|u^{j-1}\|_1^{p}\\
&\qquad\le
C \|\tilde{Q}^{\frac12}\|^p_{\HS}T^{p/2-1} \Delta t\, \sum_{j=0}^{n-1}\left(\frac12+\frac{1}{2}\mathbb{E}\sup_{l\in I_j}\|u^l\|^{2p}_1\right)\\
&\qquad=C\|A^{\frac12}Q^{\frac12}\|^{p}_{\HS}T^{p/2}+ C\|A^{\frac12}Q^{\frac12}\|^{p}_{\HS}T^{p/2-1}\Delta t\, \sum_{j=0}^{n-1}\left(\mathbb{E}\sup_{l\in I_j}\|u^l\|^{2p}_1\right).
\end{aligned}
\end{equation}
Finally, substituting \eqref{eq:es2} and \eqref{eq:es3} into
\eqref{eq:es1} yields the desired bound by using the discrete Gronwall
inequality. Before applying the discrete Gronwall inequality we kick
back the last term from the sum $T^{p-1}\Delta
t\sum_{j=1}^n\mathbb{E}\sup_{l\in I_j}\|u^l\|^{2p}_1$ in \eqref{eq:es1} using the
condition $T^{p-1}\Delta t\le \frac12$.
\end{proof}
\section{The convergence results}\label{sec:mr}

We begin by showing a maximal type error estimate for the linear problem. Define the
Backward Euler approximation of the stochastic convolution $W_A(t_n)$ by
$$
W_A^n:=\sum_{k=1}^nE^{n-k+1}\Delta W^k=\sum_{k=1}^n\int_{t_{k-1}}^{t_k}E^{n-k+1}\,\dd W(s),
\ \text{where $E^n=(I+\Delta t\, A)^{-n}$.}
$$

The following result has been proved in a larger generality for
multiplicative noise in Banach spaces using heavy machinery in the
range $0\le \beta<1$. This would be enough for the purposes of the
semilinear problem with additive noise.  However, it is possible to
obtain the range $0\le \beta \le 2$ because the noise is additive and
the approximation of the noise is exact at the mesh points. Since this
result is interesting on its own, and the proof presented here is
rather elementary based on a discrete version of the factorization
method, we present the result and the proof for the full range $0\le
\beta \le 2$.

\begin{prop}\label{prop:wa}
  Let $\varepsilon\in (0,\frac12)$, $p>\frac{1}{\varepsilon}$, $0\le \beta \le 2$,
 and $T=N\Delta t$. Then there is $C=C(p,\varepsilon,T)$ such that
$$
\left(\mathbb{E}\sup_{t_n\in [0,T]}\|W_A(t_n)-W_A^n\|^p\right)^\frac{1}{p}\le
C\Delta t^{\frac{\beta
    }{2}}\|A^{\frac{\beta-1}{2}+\varepsilon}Q^{\frac12}\|_{\HS},\quad
t_n=n\Delta t.
$$
\end{prop}
\begin{proof}
  Define the deterministic error operator $F_n$ by
  $F_n=E(t_n)-E^n$. It is well known that the following error estimate
  holds
\begin{equation}\label{eq:dete}
\|A^{\frac{\rho}{2}}F_nv\|\le C\Delta t^\frac{\beta}{2}
t_n^{-\frac{\beta-\gamma+\rho}{2}}\|A^{\frac{\gamma}{2}}v\|,
\quad 0\le \gamma\le \beta+\rho,~\rho,\gamma\ge 0,~\beta\in [0,2].
\end{equation}
Next, we consider the decomposition
\begin{align*}
W_A(t_n)-W_A^n&=\sum_{k=1}^n\int_{t_{k-1}}^{t_k}(E(t_n-\sigma)-E^{n-k+1})\,\dd W(\sigma)\\
&=\sum_{k=1}^n\int_{t_{k-1}}^{t_k}(E(t_n-\sigma)-E(t_{n-k+1}))\,\dd W(\sigma)\\
&\quad+\sum_{k=1}^n\int_{t_{k-1}}^{t_k}(E(t_{n-k+1})-E^{n-k+1})\,\dd W(\sigma)=:e^n_1+e^n_2.
\end{align*}
To estimate $e_1$ we first write
\begin{align*}
e^n_1=\sum_{k=1}^n\int_{t_{k-1}}^{t_k}E(t_n-\sigma)(I-E(\sigma-t_{k-1}))\,\dd
W(\sigma)
=\int_0^{t_n}E(t_n-\sigma)\Psi(\sigma)\,dW(\sigma)
\end{align*}
with $\Psi(\sigma)=(I-E(\sigma-t_{k-1}))$ for $\sigma\in (t_{k-1},t_k]$.
Next we use the factorization method from \cite[Chapter 5]{DPZ} to write
\begin{align*}
e^n_1&=c_{\alpha}\int_0^{t_n}E(t_n-\sigma)\int_{\sigma}^{t_n}(t_n-s)^{-1+\alpha}(s-\sigma)^{-\alpha}\,\dd s\,\dd W(\sigma)\\
&=c_{\alpha}\int_0^{t_n}(t_n-s)^{-1+\alpha}E(t_n-s)\int_0^s(s-\sigma)\Psi(\sigma)E(s-\sigma)\,\dd W(\sigma)\,\dd s\\
&=c_{\alpha}\int_0^{t_n}(t_n-s)^{-1+\alpha}E(t_n-s)Y(s)\,\dd s,
\end{align*}
where $\alpha\in (0,\frac12)$, $c_\alpha^{-1}=\int_{\sigma}^t(t-s)^{-1+\alpha}(s-\sigma)^{-\alpha}\,\dd s$ and
$$
Y(s)=\int_0^s(s-\sigma)\Psi(\sigma)E(s-\sigma)\,\dd W(\sigma).
$$
Therefore, by H\"older's inequality and that $\|E(t)\|\le 1$ for all $t\geq 0$,
\begin{align*}
\mathbb{E}\sup_{t_n\in [0,T]}\|e_2^n\|^p\le c_\alpha \left(\int_0^T s^{(-1+\alpha)\frac{p}{p-1}}\,\dd s\right)^{p-1}\int_0^T\mathbb{E}\|Y(s)\|^p\,\dd s.
\end{align*}
The first integral is finite for $p>\frac{1}{\alpha}$. To bound the
second integral, notice that $Y(s)$ is a Gaussian random variable for
all $s\in [0,T]$ and therefore, we use \eqref{eq:es1a} to bound its
$p$-th moment, \eqref{eq:hs}, \eqref{eq:anal0} with $\beta=\frac12
-\varepsilon$, and \eqref{eq:anal} with $\beta=\gamma=0$ and
$\rho=\frac{\beta}{2}$, to obtain
\begin{align*}
\mathbb{E}\|Y(s)\|^p&\le C_p\left(\int_0^s(s-\sigma)^{-2\alpha}\|\Psi(\sigma)E(s-\sigma)Q^{\frac12}\|^2_{\HS}\,\dd \sigma\right)^{\frac{p}{2}}\\
&=C_p\left(\int_0^s(s-\sigma)^{-2\alpha}\|\Psi(\sigma)A^{-\frac{\beta}{2}}A^{\frac12-\varepsilon}E(s-\sigma)A^{\frac{\beta-1}{2}+\varepsilon}Q^{\frac12}\|^2_{\HS}\,\dd \sigma\right)^{\frac{p}{2}}\\
&\le C_p\|A^{\frac{\beta-1}{2}+\varepsilon}Q^{\frac12}\|_{\HS}^p\left(\int_0^s(s-\sigma)^{-2\alpha -1+2\varepsilon}\|\Psi(\sigma)A^{-\frac{\beta}{2}}\|^2\,\dd \sigma\right)^{\frac{p}{2}}
\\
&\le C\Delta t^{\frac{\beta p}{2}}\|A^{\frac{\beta-1}{2}+\varepsilon}Q^{\frac12}\|_{\HS}^p\left(\int_0^s(s-\sigma)^{-2\alpha -1+2\varepsilon}\,\dd s \right)^{\frac{p}{2}}\\
&\le C_{T,p,\alpha,\varepsilon}\Delta t^{\frac{\beta p}{2}}\|A^{\frac{\beta-1}{2}+\varepsilon}Q^{\frac12}\|_{\HS}^p,
\end{align*}
provided that $\alpha<\varepsilon$. Given $p>1/\varepsilon$, we thus
need to choose $\alpha\in(\frac1p,\varepsilon)$. We conclude
$$
\int_0^T\mathbb{E}\|Y(s)\|^p\,\dd s\le TC_{T,p,\alpha,\varepsilon}\Delta t^{\frac{\beta p}{2}}\|A^{\frac{\beta-1}{2}+\varepsilon}Q^{\frac12}\|_{\HS}^p,
$$
which proves the bound on $e_1^n$.  To bound $e_2^n$ we use a
discrete version of the factorization method. First introduce the
constants
$$
c_{n,k}:=\Big(\Delta t\,\sum_{l=k}^nt_{n-l+1}^{-1+\alpha}t_{l-k+1}^{-\alpha}\Big)^{-1}.
$$
It is not difficult to see that $c_{n,k}\le C$ for all $1\le k\le n$. Then we have
\begin{align*}
e_2^n&=\sum_{k=1}^nE(t_{n-k+1})c_{n,k}\left(\Delta t\,\sum_{l=k}^nt_{n-l+1}^{-1+\alpha}t_{l-k+1}^{-\alpha}\right)\Delta W^k\\
&\quad-\sum_{k=1}^nE^{n-k+1}c_{n,k}\left(\Delta t\,\sum_{l=k}^nt_{n-l+1}^{-1+\alpha}t_{l-k+1}^{-\alpha}\right)\Delta W^k\\
&=\Delta t\, \sum_{l=1}^nt_{n-l+1}^{-1+\alpha}E(t_{n-l})\sum_{k=1}^l c_{n,k} t_{l-k+1}^{-\alpha} E(t_{l-k+1})\Delta W^k\\
&\quad - \Delta t\, \sum_{l=1}^nt_{n-l+1}^{-1+\alpha}E^{n-l}\sum_{k=1}^l c_{n,k} t_{l-k+1}^{-\alpha} E^{l-k+1}\Delta W^k\\
&=\Delta t\, \sum_{l=1}^nt_{n-l+1}^{-1+\alpha}E(t_{n-l})Y^l-\Delta t\, \sum_{l=1}^nt_{n-l+1}^{-1+\alpha}E^{n-l}\tilde{Y}^l\\
&=\Delta t\, \sum_{l=1}^nt_{n-l+1}^{-1+\alpha}F_{n-l}Y^l+\Delta t\,\sum_{l=1}^n t_{n-l+1}^{-1+\alpha} E^{n-l}(Y^l-\tilde{Y}^l)=:e^n_{21}+e^n_{22},
\end{align*}
where
$$
Y_l=\sum_{k=1}^l c_{n,k} t_{l-k+1}^{-\alpha} E(t_{l-k+1})\Delta W^k,
\quad
\tilde{Y}_l=\sum_{k=1}^l c_{n,k} t_{l-k+1}^{-\alpha} E^{l-k+1}\Delta W^k.
$$
Next, we bound $e^n_{21}$, by H\"older's inequality and \eqref{eq:dete} with $\rho=0$ and $\gamma=\beta$, as follows
\begin{align*}
\mathbb{E}\sup_{t_n\in [0,T]}\|e^n_{21}\|^p
&\le  \left(\Delta t\, \sum_{l=1}^N\big(t_{l}^{-1+\alpha}\|F_{l}A^{-\frac{\beta}{2}}\|\big)^{\frac{p}{p-1}}\right)^{p-1}\mathbb{E}\Delta t\,\sum_{l=1}^N\|A^{\frac{\beta}{2}}Y^l\|^p\\
&\le C \Delta t^{\frac{\beta p}{2}} \left(\Delta t\, \sum_{l=1}^nt_l^{(-1+\alpha)\frac{p}{p-1}}\right)^{p-1} \mathbb{E}\Delta t\,\sum_{l=1}^N\|A^{\frac{\beta}{2}}Y^l\|^p,
\end{align*}
where the first sum is finite if $p>\frac{1}{\alpha}$. To estimate
the last sum, note that $A^{\beta/2}Y^l$ is a Gaussian random variable
and hence, as before, we use \eqref{eq:es1a} to bound its $p$-th
moment. Therefore, using also \eqref{eq:hs} and \eqref{eq:anal0},
\begin{align*}
&\mathbb{E}\Delta t\,\sum_{l=1}^N\|A^{\frac{\beta}{2}}Y^l\|^p=\Delta t\,\sum_{l=1}^N\mathbb{E}\left\|\left(\sum_{k=1}^l c_{n,k} t_{l-k+1}^{-\alpha} A^{\frac{\beta}{2}}E(t_{l-k+1})\Delta W^k\right)\right\|^p\\
&\qquad =\Delta t\,\sum_{l=1}^N \left(\Delta t\,\sum_{k=1}^l c^2_{n,k} t_{l-k+1}^{-2\alpha} \|A^{\frac{\beta}{2}}E(t_{l-k+1})Q^{\frac12}\|_{\HS}^2\right)^{\frac{p}{2}} \\
&\qquad \le C \Delta t\,\sum_{l=1}^N \left(\Delta t\,\sum_{k=1}^N t_{k}^{-2\alpha} \|E(t_{k})A^{\frac12-\varepsilon}A^{\frac{\beta-1}{2}+\varepsilon}Q^{\frac12}\|_{\HS}^2\right)^{\frac{p}{2}}\\
&\qquad \le CT \|A^{\frac{\beta-1}{2}+\varepsilon}Q^{\frac12}\|^p_{\HS}\left(\Delta t\, \sum_{k=1}^N t_{k}^{-1-2\alpha+2\varepsilon}\right)^{\frac{p}{2}}\le C_{T,p,\alpha,\varepsilon} \|A^{\frac{\beta-1}{2}+\varepsilon}Q^{\frac12}\|^p_{\HS},
\end{align*}
provided that $\alpha<\varepsilon$. Finally, we estimate
$e_{22}^n$. By H\"older's inequality we first get
\begin{align*}
&\mathbb{E}\sup_{t_n\in [0,T]}\|e_{22}^n\|^p\le \left(\Delta t\,\sum_{l=1}^N t_{l}^{-1+\alpha} \|E^{l}\|^{\frac{p}{p-1}}\right)^{p-1}\mathbb{E} \sum_{l=1}^N\|Y(l)-\tilde{Y}(l)\|^p\\
&\qquad \le  \left(\Delta t\,\sum_{l=1}^N t_{l}^{(-1+\alpha)\frac{p}{p-1}}\right)^{p-1}\mathbb{E} \sum_{l=1}^N\|Y(l)-\tilde{Y}(l)\|^p\le C_{\alpha,p}\sum_{l=1}^N\|Y(l)-\tilde{Y}(l)\|^p,
\end{align*}
if $p>\frac{1}{\alpha}$. To estimate the last term, we use
\eqref{eq:es1a} to bound the $p$-th moment of a Gaussian random
variable and also \eqref{eq:hs} and \eqref{eq:dete} with $\rho=1-2\varepsilon$ and $\gamma=\beta$ to get
\begin{align*}
&\sum_{l=1}^N\|Y(l)-\tilde{Y}(l)\|^p=\sum_{l=1}^N\left\|\sum_{k=1}^l c_{n,k} t_{l-k+1}^{-\alpha}F_{l-k+1}\Delta W^k\right\|^{p}\\
&\qquad\le C_p \sum_{l=1}^N \left(\Delta t\, \sum_{k=1}^N t_k^{-2\alpha} \|F_{k}Q^{\frac12}\|_{\HS}^2\right)^{\frac{p}{2}}\\
&\qquad=C_p \sum_{l=1}^N \left(\Delta t\, \sum_{k=1}^N t_k^{-2\alpha} \|A^{\frac{1}{2}-\varepsilon}F_{k}A^{-\frac{\beta}{2}}A^{\frac{\beta-1}{2}+\varepsilon}Q^{\frac12}\|_{\HS}^2\right)^{\frac{p}{2}}\\
&\qquad \le C_p \Delta t^{\frac{\beta p}{2}} \left(\Delta t\, \sum_{k=1}^N t_k^{-1-2\alpha+2\varepsilon}\right)^{\frac{p}{2}}\|A^{\frac{\beta-1}{2}+\varepsilon}Q^{\frac12}\|^p_{\HS}\\
&\qquad \le C_{T,p,\alpha,\varepsilon}\Delta t^{\frac{\beta p}{2}}\|A^{\frac{\beta-1}{2}+\varepsilon}Q^{\frac12}\|^p_{\HS},
\end{align*}
whenever $\alpha<\varepsilon$, which finishes the proof.
\end{proof}
Next we state a Lipschitz estimate for $f(u)$.  Here we use Sobolev's
inequality and, similarly to \eqref{eq:lg}, it is crucial that $d\le 3$ and that the nonlinearity
$f$ is at most cubic.  For a proof we refer to \cite[Lemma 2.5]{KML}.

\begin{lem}\label{lem:lip}
For all $u,v\in \dot{H}^1$ we have
$$
\|A^{-\frac{1}{2}}(f(u)-f(v))\|\le C(\|u\|^2_1+\|v\|^2_1)\|u-v\|.
$$
\end{lem}
We are now ready to state and prove the pathwise convergence of the Backward Euler scheme defined in \eqref{eq:be}.

\begin{thm}\label{thm:pwc}
Let $\varepsilon>0$, $\|A^{\frac12+\varepsilon} Q^{\frac12}\|_{\HS}<\infty$, $\mathbb{E}\|u_0\|_1^2<\infty$, $0\le \gamma <\frac12$, and $T=N\Delta t$. Then, there are finite random variables $K\ge 0$ and $\Delta t_0>0$ such that, almost surely,
$$
\sup_{t_n\in [0,T]}\|u(t_n)-u^n\|\le K \Delta t^\gamma,\quad t_n=n\Delta t,\quad \Delta t\le \Delta t_0.
$$
\end{thm}
\begin{proof}
  Since the arguments are pathwise and hence basically deterministic,
  we omit standard details. Let $e^n=u(t_n)-u^n$ and $0\le \gamma
  <\frac12$. We decompose the error, using the mild formulation of
  \eqref{eq:be} and \eqref{eq:mild}, as follows
\begin{align*}
e^n&=(E(t_n)u_0-E^n u_0)
+(W_A(t_n)-W_A^n)\\
&\quad +\sum_{k=1}^n
\int_{t_{k-1}}^{t_k}
E(t_n-s)f(u(s))-E^{n-k+1}f(u^{k})\,\dd s
=:e^n_1+e^n_2+e^n_3.
\end{align*}
By \eqref{eq:dete} we may estimate $e_1$ as
$$
\|e^n_1\|\le C \Delta t^{\frac12}\|u_0\|_1.
$$
For $e_2^n$, by Proposition \ref{prop:wa} with $\beta=2$, we have that
$$
\|e^n_2\|\le L \Delta t\, \|A^{\frac12+\varepsilon} Q^{\frac12}\|_{\HS}
$$
almost surely for some finite nonnegative random variable $L$. Next, we can further decompose $e_3$ as
\begin{align*}
e^n_3&=\sum_{k=1}^n\int_{t_{k-1}}^{t_k}E^{n-k+1}(f(u(t_{k}))-f(u^{k}))\,\dd s\\
&\quad +\sum_{k=1}^n\int_{t_{k-1}}^{t_k}(E(t_{n-k+1})-E^{n-k+1})f(u(t_{k}))\,\dd s\\
&\quad +\sum_{k=1}^n\int_{t_{k-1}}^{t_k}E(t_{n-k+1})(f(u(s))-f(u(t_{k})))\,\dd s\\
&\quad +\sum_{k=1}^n\int_{t_{k-1}}^{t_k}(E(t_n-s)-E(t_{n-k+1}))f(u(s))\,\dd s
=:e_{31}^n+e_{32}^n+e_{33}^n+e_{34}^n.
\end{align*}
To bound $e_{31}^n$ we use Propositions \ref{prop:eb} and
\ref{prop:ul} together with Lemma \ref{lem:lip} to conclude that for
some finite nonnegative random variable $L_1$ we have, almost surely,
\begin{align*}
\|e_{31}^n\|
&=\left\|\sum_{k=1}^n\int_{t_{k-1}}^{t_k}A^{\frac{1}{2}}E^{n-k+1}A^{-\frac12}(f(u(t_{k}))-f(u^{k}))\,\dd s\right\|\\
&\le L_1\sum _{k=1}^n\int_{t_{k-1}}^{t_k} t_{k}^{-\frac 12} \|e^k\|\,\dd s =L_1 \Delta t\, \sum _{k=1}^n t_{k}^{-\frac 12} \|e^k\|,
\end{align*}
where we used the well known fact that $\|A^{1/2}E^k\|\le
Ct_k^{-\frac12}$ (see, for example, \cite[Lemma
7.3]{Thomeebook}). Next we use Proposition \ref{prop:eb}, Lemma
\ref{lem:lip}, and \eqref{eq:dete} with $\gamma=0$, $\rho=1$ and $\beta=2\gamma$ to estimate $e_{32}^n$ as
\begin{align*}
\|e_{32}^n\|
&=\left\|\sum_{k=1}^n\int_{t_{k-1}}^{t_k}A^{\frac12}(E(t_{n-k+1})-E^{n-k+1})A^{-\frac12}f(u(t_{k}))\,\dd s\right\|\\
&\le\Delta t^{\gamma} L_2 \sum_{k=1}^n\int_{t_{k-1}}^{t_k}t_k^{-\frac12-\gamma}\,\dd s=\Delta t^{\gamma} L_2 \Delta t\,\sum_{k=1}^n t_k^{-\frac12-\gamma},
\end{align*}
almost surely for some finite nonnegative random variable $L_2$. For $e_{33}^n$ we use the H\"older continuity of $u$ from Proposition \ref{propo:ulip} together with Proposition \ref{prop:eb}, Lemma \ref{lem:lip}, and \eqref{eq:anal0} with $\beta=\frac12$,
and obtain
\begin{align*}
\|e_{33}^n\|
&=\left\|\sum_{k=1}^n\int_{t_{k-1}}^{t_k}A^{\frac12}E(t_{n-k+1})A^{-\frac12}(f(u(s))-f(u(t_{k})))\,\dd s\right\|\\
&\le \Delta t^{\gamma} L_3\sum_{k=1}^n\int_{t_{k-1}}^{t_k}t_k^{-\frac12}\,\dd s=\Delta t^{\gamma} L_3 \Delta t\, \sum_{k=1}^nt_k^{-\frac12},
\end{align*}
almost surely for some finite nonnegative random variable $L_3$. Finally, by Proposition \ref{prop:eb}, Lemma \ref{lem:lip} and \eqref{eq:anal} with $\beta=\frac12$ and $\rho=0$, we have
\begin{align*}
\|e_{34}\|
&=\left\|\sum_{k=1}^n\int_{t_{k-1}}^{t_k}A^{\frac12}(E(t_n-s)-E(t_{n-k+1}))A^{-\frac12}f(u(s))\,\dd s\right\|\\
&\le \Delta t^{\gamma} L_4 \sum_{k=1}^n\int_{t_{k-1}}^{t_k}t_k^{-\frac12-\gamma}\,\dd s=\Delta t^{\gamma} L_4 \Delta t\,\sum_{k=1}^n t_k^{-\frac12-\gamma},
\end{align*}
almost surely for some finite nonnegative random variable $L_2$. Putting together the estimates and using  a generalized discrete Gronwall lemma \cite[Lemma 7.1]{Stig} finishes the proof.
\end{proof}

Finally, we show strong convergence in $L^p$, albeit without rate.

\begin{thm}\label{thm:strc}
  Let $\varepsilon>0$, $p\ge 1$, and $N\Delta t=T$. If
  $\|A^{\frac12+\varepsilon} Q^{\frac12}\|_{\HS}<\infty$ and
  $\mathbb{E}\|u_0\|_1^{2p}<\infty$, then
\begin{align*}
\lim_{\Delta t\to 0}\mathbb{E}\sup_{t_{n}\in [0,T]}\|u(t_n)-u^n\|^p= 0,\quad t_n=n\Delta t.
\end{align*}
\end{thm}

\begin{proof}
Let
\begin{align*}
Y_N:=\sup_{t_{n}\in [0,T]}\|u(t_n)-u^n\|^p.
\end{align*}
By Theorem \ref{thm:pwc} it follows that $Y_N\to 0$ almost surely, and
hence in probability, as $N\to \infty$.  By Propositions
\ref{prop:eb} and \ref{prop:ulp}, there is $M>0$ such that, for
$T^{2p-1}\Delta t\le \frac12$,
\begin{align*}
\mathbb{E}\, Y_N^{2}\le C\mathbb{E}\sup_{t_{n}\in [0,T]}\left(\|u(t_n)\|^{2p}+\|u^n\|^{2p}\right)\le M.
\end{align*}
Therefore, it follows that $\{Y_N\}_{N\in \mathbb{N}}$ is uniformly
integrable. Being convergent in probability and uniformly integrable,
it converges in $L^1$; that is,
\begin{align*}
\lim_{\Delta t\to 0}\mathbb{E}\sup_{t_{n}\in [0,T]}\|u(t_n)-u^n\|^p=\lim_{N\to \infty}\mathbb{E}\,Y_N= 0,
\end{align*}
see \cite[Proposition 3.12]{Kal}.
\end{proof}


\def\cprime{$'$}
\providecommand{\bysame}{\leavevmode\hbox to3em{\hrulefill}\thinspace}
\providecommand{\MR}{\relax\ifhmode\unskip\space\fi MR }
\providecommand{\MRhref}[2]{%
  \href{http://www.ams.org/mathscinet-getitem?mr=#1}{#2}
}
\providecommand{\href}[2]{#2}

\end{document}